\documentclass[11pt,reqno]{amsart}

\usepackage{graphicx}
\usepackage{amsmath,amssymb,mathrsfs,amsfonts,cite}
\usepackage{color}
\usepackage[colorlinks=true,linkcolor=blue]{hyperref}
\usepackage{bm}
\usepackage{tikz}
\newtheorem{theorem}{Theorem}[section]

\newtheorem{lemma}[theorem]{Lemma}

\theoremstyle{definition}

\theoremstyle{remark}

\numberwithin{equation}{section}

\begin{document}

\title[an adaptive FEM DtN method]
{An adaptive finite element DtN method for the three-dimensional acoustic
  scattering problem}

\author{Gang Bao}
\address{School of Mathematical Science, Zhejiang University,
Hangzhou 310027, China.}
\email{baog@zju.edu.cn}

\author{Mingming Zhang}
\address{School of Mathematical Science, Zhejiang University,
Hangzhou 310027, China.}
\email{mmzaip@zju.edu.cn}

\author{Bin Hu}
\address{School of Mathematical Science, Zhejiang University,
Hangzhou 310027, China.}
\email{binh@zju.edu.cn}

\author{Peijun Li}
\address{Department of Mathematics, Purdue University, West Lafayette,
Indiana 47907, USA.}
\email{lipeijun@math.purdue.edu}

\thanks{The work of GB is supported in part by an NSFC Innovative Group Fund
(No.11621101). The research of PL is supported
in part by the NSF grant DMS-1912704.}

\subjclass[2010]{65M30, 78A45, 35Q60}

\keywords{acoustic scattering problem, adaptive finite element method,
transparent boundary condition, a posteriori error estimates}

\begin{abstract}
This paper is concerned with a numerical solution of the acoustic scattering by
a bounded impenetrable obstacle in three dimensions. The obstacle scattering
problem is formulated as a boundary value problem in a bounded domain by using
a Dirichlet-to-Neumann (DtN) operator. An a posteriori error estimate is
derived for the finite element method with the truncated DtN operator. The a
posteriori error estimate consists of the finite element approximation error
and the truncation error of the DtN operator, where the latter is shown to decay
exponentially with respect to the truncation parameter. Based on the a
posteriori error estimate, an adaptive finite element method is
developed for the obstacle scattering problem. The truncation parameter is
determined by the truncation error of the DtN operator and the mesh elements for
local refinement are marked through the finite element approximation error.
Numerical experiments are presented to demonstrate the effectiveness of the
proposed method. 
\end{abstract}

\maketitle

\section{Introduction}

Wave scattering by bounded impenetrable media is usually referred to as the
obstacle scattering problem. It has played an important role in many scientific
areas such as radar and sonar, non-destructive testing, medical imaging, and
geophysical exploration \cite{CK98}. Due to the significant applications, the
obstacle scattering problem has been extensively studied in the past several
decades. Consequently, a variety of methods have been developed to solve the
scattering problem mathematically and numerically such as the method of boundary
integral equations \cite{CK83, N01} and the finite element method \cite{J93,
M03}. This paper concerns a numerical solution of the acoustic wave scattering
by an obstacle in three dimensions.  

As an exterior boundary value problem, the obstacle scattering problem is
formulated in an open domain, which needs to be truncated into a bounded
computational domain when applying numerical methods such as the finite element
method. It is indispensable to impose a boundary condition on the boundary of
the truncated domain. The ideal boundary condition is to completely avoid
artificial wave reflection by mimicking the wave propagation as if the boundary
did not exist \cite{BT80}. Such a boundary condition is called an absorbing
boundary condition \cite{EM77}, a nonreflecting boundary condition \cite{GK95},
or a transparent boundary condition (TBC) \cite{GK04}. It still remains as an
active research topic in
computational wave propagation \cite{H99}, especially for time-domain scattering
problems \cite{BGL18}. Since Berenger proposed the perfectly matched layer (PML)
technique for the time-domain Maxwell equations \cite{B94}, the PML method has
been extensively studied for various wave propagation problems \cite{BW05,
CM98, TY98}. As an effective approach for the domain truncation, the basic idea
of the PML technique is to surround the domain of interest by a layer of finite
thickness with specially designed artificial medium that would attenuate all the
waves coming from inside of the domain. Combined with the PML technique, the a
posteriori error estimate based adaptive finite element methods were developed
for the diffraction grating problems \cite{CW03, BLW10} and the obstacle
scattering problems \cite{CL05}. It was shown that the estimates consist of the
finite element discretization error and the PML truncation error which has an
exponential rate of convergence with respect to the PML parameters. 

Recently, an alternative adaptive finite element method was developed for
solving the two-dimensional acoustic obstacle scattering problem \cite{JLZ13},
where the PML was replaced by the TBC to truncate the open domain. Since the
TBC is exact, it can be imposed on the boundary which could be put as close as
possible to the obstacle. Hence it does not require an extra absorbing layer of
artificial
medium to enclose the domain of interest. Based on a nonlocal
Dirichlet-to-Neumann (DtN) operator, the TBC is given as an infinite
Fourier series. Practically, the series needs to be truncated into a sum of
finitely many, say $N$, terms, where $N$ is an appropriately chosen positive
integer. In \cite{JLZ13}, an a posteriori error estimate was derived for the
finite element discretization
but it did not include the truncation error of the DtN operator. The complete a
posteriori error estimate was obtained in \cite{JLLZ-JSC17}. The new estimate
takes into account both the finite element discretization error and the DtN
operator truncation error. It was shown that the truncation error decays
exponentially with respect to the truncation parameter $N$. The adaptive finite
element DtN method has also been applied to solve the diffraction grating
problems \cite{WBLLW15} as well as the elastic wave equation in periodic
structures \cite{LY20}. The numerical results show that the adaptive finite
element DtN method is competitive with the adaptive finite element PML method. 

In this work, we extend the analysis in \cite{JLLZ-JSC17} to the
three-dimensional obstacle scattering problem. It is worthy to mention that
the extension is nontrivial since more complex spherical Hankel functions need
to be considered and the computation is more challenging in three dimensions.
Specifically, we consider the acoustic wave scattering by a sound hard obstacle.
Based on a TBC, the exterior problem is formulated equivalently into a boundary
value problem in a bounded domain for the three-dimensional Helmholtz equation.
Using a duality argument, we derive the a posteriori error estimate which
includes the finite element discretization error and the DtN operator truncation
error. Moreover, we show that the truncation error has an exponential rate of
convergence with respect to the truncation parameter $N$. The a posteriori
estimate is used to design the adaptive finite element algorithm to choose
elements for refinements and to determine the truncation parameter $N$. In
addition, we present a technique to deal with adaptive mesh refinements of the
surface. Numerical experiments are included to demonstrate the effectiveness of
the proposed method.

This paper is organized as follows. In Section 2, we introduce the model problem
of the acoustic wave scattering by an obstacle in three dimensions. The
variational formulation is given for the boundary value problem by using the
DtN operator. In Section 3, we present the finite element approximation with the
truncated DtN operator. Section 4 is devoted to the a posteriori error analysis
by using a duality argument. In Section 5, we discuss the numerical
implementation and the adaptive finite element DtN method, and present two
numerical examples to demonstrate the effectiveness of the proposed method. The
paper is concluded with some general remarks and directions for future work in
Section 6.

\section{Problem formulation}

Consider a bounded sound-hard obstacle $D$ with Lipschitz continuous
boundary $\partial D$ in $\mathbb R^3$. Denote by
$B_r=\{x\in\mathbb{R}^3:|x|<r\}$ the ball which is centered at the origin and
has a radius $r$. Let $R$ and $R'$ be two positive constants
such that $R>R'>0$ and $\overline{D}\subset B_{R'}\subset B_{R}$. Denote
$\Omega=B_R\backslash\overline{D}$. The obstacle scattering problem for
acoustic waves can be modeled by the following exterior boundary value problem: 
\begin{equation}\label{eq}
  \begin{cases}
\Delta u + \kappa^2 u = 0 \quad &{\rm in} ~
    \mathbb{R}^3\setminus\overline{D}, \\
\partial_\nu u = -g \quad &{\rm on } ~\partial D, \\
\lim\limits_{r\rightarrow \infty} r(\partial_r u
     - {\rm i}\kappa u) = 0, &r=|x|,
  \end{cases}
\end{equation}
where $\kappa>0$ is the wavenumber and $\nu$ is the unit outward normal 
vector to $\partial D$. Although the results are given for the sound-hard
boundary condition in this paper, the method can be applied to other types of
boundary conditions, such as the sound-soft and impedance boundary conditions.

Let
  $\hat{x}_1=\sin\theta\cos\varphi$, $\hat{x}_2=\sin\theta\sin\varphi$, $\hat{x}_3=\cos\theta$,
  $\theta\in [0,\pi]$ and $\varphi\in [0,2\pi]$. Introduce the spherical
harmonic functions
\begin{equation*}
  Y_n^m(\hat{x})=Y_n^m(\theta,\varphi)
  =\sqrt{\frac{(2n+1)(n-|m|)!}{4\pi(n+|m|)!}}P_n^{|m|}(\cos\theta)
  e^{{\rm i}m\varphi}, \quad m=-n,\dots, n,\, n=0, 1, \dots,
  \end{equation*}
where \[
P_n^m(t)=(1-t^2)^{\frac{m}{2}}\frac{{\rm d}^m}{{\rm
d}t^m}P_n(t),\quad -1\leq t\leq 1,
\]
are called the associated Legendre functions and $P_n$ are the Legendre 
polynomials. It is known that the spherical harmonic functions $\{Y_n^m: 
m=-n,\dots, n,\, n=0, 1, \dots\}$ form an orthonormal system in
$L^2(\mathbb{S}^2)$, where $\mathbb S^2=\{x\in\mathbb R^3: |x|=1\}$ is the unit
sphere in $\mathbb R^3$. For any function $u\in L^2(\partial B_R)$, it admits
the Fourier series expansion
\[
u(x):=u(R, \hat
x)=\sum_{n=0}^{\infty}\sum_{m=-n}^{m=n}\hat{u}_n^m(R)Y_n^m(\hat{x}),\quad
\hat{u}_n^m=\int_{\mathbb{S}^2}u(R,\hat x)\bar{Y}_n^m(\hat x)
  {\rm d}\hat x. 
\]

Using the Fourier coefficients, we may define an equivalent $L^2(\partial B_R)$
norm of $u$ as
\[
\|u\|_{L^2(\partial
B_R)}=\left(\sum_{n=0}^{\infty}\sum_{m=-n}^{n}|\hat{u}_n^m|^2\right)^{\frac{1}{2
}}.
\]
 The trace space $H^{s}(\partial B_R)$ is defined by
\[
H^{s}(\partial B_R)=\left\{u\in L^2(\partial B_R) ~
:~\|u\|_{H^{s}(\partial B_R)}<\infty\right\},
\]
where the norm may be characterized by 
\begin{equation}\label{seminorm}
\|u\|_{H^{s}(\partial B_R)}^2
=\sum_{n=0}^{\infty}\sum_{m=-n}^{m=n}\left(1+n(n+1)\right)^s|\hat{u}_n^m|^2.
\end{equation}
Clearly, the dual space of $H^{-s}(\partial B_R)$ is
$H^{s}(\partial B_R)$ with respect to the scalar product in
$L^2(\partial B_R)$ defined by
\[
\langle u,v\rangle_{\partial B_R}=\int_{\partial B_R}u\bar{v}{\rm d}s.
\]

In the exterior domain $\mathbb{R}^3\backslash \overline{B}_R$, the
solution of the Helmholtz equation in (\ref{eq}) can be written as
\begin{equation}\label{exp}
  u(r,\hat x)
  =\sum_{n=0}^{\infty}\sum_{m=-n}^{n}\hat{u}_n^m\frac{h_n^{(1)}(\kappa
    r)}{h_n^{(1)}(\kappa R)}Y_n^m(\hat x), \quad r>R,
  \end{equation}
where $h_n^{(1)}$ is the spherical Hankel function of the first
kind with order $n$ and is defined as (cf. \cite{HH-cmame15})
\begin{equation*}
h_n^{(1)}(z)=\sqrt{\frac{\pi}{2z}}H_{n+\frac{1}{2}}^{(1)}(z).
\end{equation*}
Here $H_{n+\frac{1}{2}}^{(1)}(\cdot)$ is the Hankel function of
the first kind with order $n+\frac{1}{2}$. 

Define the DtN operator $T: H^{\frac{1}{2}}(\partial B_R)
\rightarrow H^{-\frac{1}{2}}(\partial B_R)$ by
\begin{equation}\label{tu}
  (Tu)(R, \hat x)=\frac{1}{R}\sum_{n=0}^{\infty}\Theta_n(\kappa
  R)\sum_{m=-n}^{n}\hat{u}_n^m Y_n^m(\hat x),
\end{equation}
where 
\begin{equation*}
\Theta_n(z)=z\frac{{h_n^{(1)}}^{'}(z)}{h_n^{(1)}(z)},
\end{equation*}
which satisfies (cf. \cite{FNS07, HH-cmame15}):
\begin{equation}\label{pro}
\Re\Theta_n(z)\leq -\frac{1}{2}, \quad \Im\Theta_n(z)>0,\quad
\Theta_n(z)\sim n, \quad n\rightarrow\infty.
\end{equation}

The DtN operator has the following properties. The proof is similar to that of 
\cite[Lemma 1.2]{JLZ13} and is omitted here for brevity. 

\begin{lemma}\label{estT}
The DtN operator $T: H^{\frac{1}{2}}(\partial B_R)\to
H^{-\frac{1}{2}}(\partial B_R)$  is continuous, i.e., 
\[
 \|T u\|_{H^{-\frac{1}{2}}(\partial B_R)}
  \lesssim \|u\|_{H^{\frac{1}{2}}(\partial B_R)}. 
\]
Moreover, it satisfies
\begin{equation*}
 -\Re\langle Tu, u\rangle \gtrsim \|u\|_{L^{2}(\partial B_R)}^2,\quad
 \Im\langle T u, u\rangle\geq 0.
\end{equation*}
\end{lemma}

Here $a\lesssim b$ or $a\gtrsim b$ stands for $a\leq Cb$ or $a\geq Cb$, where
$C$ is a positive constant whose specific value is not required and may be
different in the context. 

It follows from (\ref{exp})--(\ref{tu}) that we have the
transparent boundary condition
\begin{equation}\label{tbc}
  \partial_ru=Tu \quad {\rm on}~\partial B_R.
\end{equation}
The weak formulation of (\ref{eq}) is to find $u\in H^1(\Omega)$ such that
\begin{equation}\label{wf}
  a(u,v)=\langle g,v\rangle_{\partial D}  \quad\forall\, v\in H^{1}(\Omega),
\end{equation}
where the sesquilinear form $a:~H^1(\Omega)\times H^1(\Omega)\rightarrow
\mathbb{C}$ is defined by
\begin{equation*}
a(u,v)=\int_\Omega \nabla u\cdot\nabla\bar{v}{\rm d}x
-\kappa^2\int_\Omega u\bar v{\rm d}x-\langle Tu,v\rangle_{\partial B_R}
\end{equation*}
and the linear functional
\[
\langle g,v\rangle_{\partial D}=\int_{\partial D}g\bar{v}{\rm d}s.
\]

\begin{theorem}\label{uniq}
The variational problem \eqref{wf} has at most one solution.
\end{theorem}

\begin{proof}
It suffices to show that $u=0$ if $g=0$. By \eqref{wf}, we have 
\[
\int_{\Omega}(|\nabla u|^2-\kappa^2|u|^2){\rm d}x-\langle T u,u\rangle_{\partial
B_R}=0.
\]
Taking the imaginary part of the above equation yields
\[
\Im\langle T u,u\rangle_{\partial B_R} = R
\sum_{n=0}^{\infty}\sum_{m=-n}^m\Im\Theta_n(\kappa R)|\hat{u}_n^m|^2=0,
\]
which gives that $\hat{u}_n^m=0$ by (\ref{pro}). Thus we have from \eqref{exp}
and \eqref{tbc} that $u = 0$ and $\partial_r u = 0$ on $\partial B_R$. We
conclude from the Holmgren uniqueness theorem and the unique
continuation \cite{JK85} that $u=0$ on $\Omega$.
\end{proof}

\begin{theorem}
The variational problem $\rm(\ref{wf})$ admits a unique weak solution $u$ in
$H^1(\Omega)$. Furthermore, there is a positive constant $C$ depending
on $\kappa$ and $R$ such that
\begin{equation*}
\|u\|_{H^1(\Omega)}\leq C\|g\|_{L^2(\partial D)}. 
\end{equation*}
\end{theorem}

\begin{proof}
First we show that there exists $u_0\in H^1(\Omega)$ such that $\partial_\nu
u_0=g$ on $\partial D$ and
\[
\|u_0\|_{H^1(\Omega)}\leq C\|g\|_{L^2(\partial D)}.
\]
For example, $u_0$ can be chosen as the unique weak solution of the following
boundary value problem
\[
\begin{cases}
-\Delta u_0 +  u_0 = 0 \quad & \text{in} ~ \Omega, \\
\partial_\nu u_0 = -g \quad & \text{on} ~\partial D, \\
u_0=0 \quad & \text{on}~\partial B_R.
\end{cases}
\]

Next is consider the variational problem: Find $u\in H^1(\Omega)$ such that
$u-u_0\in H^1(\Omega)$ and
\[
a(u-u_0,v)=\langle g,v\rangle_{\partial B_R}-a(u_0,v) \quad \forall\, v\in
H^1(\Omega).
\]
Denote $(h,v)=\langle g,v\rangle_{\partial B_R}-a(u_0,v)$. The above problem
is equivalent to the following variational problem: Find $w\in H^1(\Omega)$ such
that
\begin{equation}\label{wh}
a(w,v)=(h,v)\quad \forall\, v\in H^1(\Omega).
\end{equation}
Let $a(w, v)=a_{+}(w, v)-\kappa^2(w, v)$, where
\[
a_{+}(w,v)=\int_\Omega \nabla w\cdot\nabla\bar v{\rm d}x-\langle
Tw,v\rangle_{\partial B_R}.
\]
By Lemma \ref{estT}, we have 

\[
|a_{+}(v,v)|\geq \|\nabla v\|_{L^2(\Omega)}^2+|\Re
\langle Tv,v\rangle_{\partial B_R}|\gtrsim \|\nabla
v\|_{L^2(\Omega)}^2+\|v\|_{L^2(\partial B_R)}^2,
\]
which implies from the Friedrichs inequality that 
\[
|a_{+}(v,v)|\gtrsim \|v\|_{H^1(\Omega)}^2\quad \forall\, v\in H^1(\Omega).
\]

Let $Q : L^2(\Omega)\rightarrow H^1(\Omega)$ be a linear map defined by
\[
a_{+}(Qw,v)=(w,v)\quad \forall\, v\in H^1(\Omega).
\]
By the Lax--Milgram lemma, we obtain that $Q$ is
bounded from $L^2(\Omega)$ to $H^1(\Omega)$, i.e., it satisfies
\begin{equation}\label{kw}
\|Q w\|_{H^1(\Omega)}\lesssim \|w\|_{L^2(\Omega)}.
\end{equation}
It is clear to note that (\ref{wh}) is equivalent to the operator equation
\begin{equation*}
(I-\kappa^2 Q)w=Q h.
\end{equation*}
Due to the compact embedding of $H^1(\Omega)$ into $L^2(\Omega)$, the
operator $Q$ is a compact from $L^2(\Omega)$ to $L^2(\Omega)$. It follows
from the Fredholm alternative theorem and Theorem \ref{uniq} that 
the operator $I-\kappa^2 Q$ has a bounded inverse. Hence we have
\begin{equation}\label{whl2}
\|w\|_{L^2(\Omega)}\lesssim \|Qh\|_{L^2(\Omega)}\lesssim \|h\|_{L^2(\Omega)}.
\end{equation}

Combining (\ref{kw}) and (\ref{whl2}) yields that
\[
\|w\|_{H^1(\Omega)}=\|Q(\kappa^2 w+h)\|_{H^1(\Omega)}\lesssim \|\kappa^2
w+h\|_{L^2(\Omega)}\lesssim \|h\|_{L^2(\Omega)}.
\]
Since $u=u_0+w$, we have 
\[
\|u\|_{H^1(\Omega)}\leq \|w\|_{H^1(\Omega)}+\|u_0\|_{H^1(\Omega)}\lesssim
\|h\|_{L^2(\Omega)}+\|u_0\|_{L^2(\Omega)}. 
\]
It is easy to deduce from the definition of $h$ that 
\[
\|u\|_{H^1(\Omega)}\lesssim \|g\|_{L^2(\Omega)},
\]
which completes the proof.
\end{proof}

By the general theory in Babuska and Aziz \cite{BA73}, there exists a constant
$\gamma>0$ depending on $\kappa$ and $R$ such that the following inf-sup
condition holds: 
\[
\sup_{0\neq v\in H^1(\Omega)}\frac{|a(u,v)|}{\|v\|_{H^1(\Omega)}}\geq
\gamma \|u\|_{H^1(\Omega)}\quad \forall\, u\in H^1(\Omega).
\]

\section{Finite element approximation}

In this section, we introduce the finite element approximation of (\ref{wf}) and
present the a posteriori error estimate, which plays an important role in the
adaptive finite element method.

Let $\mathcal{M}_h$ be a regular tetrahedral mesh of the domain $\Omega$, where
$h$ represents the maximum diameter of all the elements in $\mathcal{M}_h$. In
order to avoid using the isoparametric finite element space and discussing the
approximation error of the boundaries $\partial D$ and $\partial B_R$, we
assume for simplicity that $\partial D$ and $\partial B_R$ are polyhedral. Thus any face
$F\in\mathcal{M}_h$ is a subset of $\partial\Omega$ if it has three boundary
vertices.

Let $V_h\subset H^1(\Omega)$ be a conforming finite element space, i.e.,
\[
  V_h:=\{v_h\in C(\overline{\Omega}): v_h|_K\in P_m(K) \quad\forall~
K\in\mathcal{M}_h\},
\]
where $m$ is a positive integer and $P_m(K)$ denotes the set of all polynomials
of degree no more than $m$. The finite element approximation to \eqref{wf} is
to seek $u_h\in V_h$ satisfying
\begin{equation*}
a(u_h,v_h)=\langle g, v_h\rangle_{\partial D}\quad\forall\, v_h\in V_h.
\end{equation*}

The above variational problem involves the DtN operator $T$ defined
by an infinite series in \eqref{tu}. Practically, it is necessary to truncate
the infinite series by taking finitely many terms of the expansion in order to
apply the finite element method. Given a positive integer $N$, we define the 
truncated DtN operator 
\begin{equation*}
   (T_Nu)(R,\hat x)=\frac{1}{R}\sum_{n=0}^{N}\Theta_n(\kappa
  R)\sum_{m=-n}^{m=n}\hat{u}_n^m Y_n^m(\hat x).
\end{equation*}
Using the truncated DtN operator $T_N$, we have the truncated finite element
approximation to the problem (\ref{wf}): Find $u_h^N\in V_h$ such that 
\begin{equation}\label{fem}
a_N(u_h^N,v_h)=\langle g,v_h^N\rangle_{\partial D}  \quad\forall\, v_h\in V_h,
\end{equation}
where the sesquilinear form $a_N:~V_h\times V_h\rightarrow \mathbb{C}$ is
defined by
\begin{equation}\label{febi}
a_N(u,v)=\int_\Omega \nabla u\cdot\nabla \bar v{\rm d}x
-\kappa^2\int_\Omega u\bar v{\rm d}x-\langle T_Nu,v\rangle_{\partial
B_R}.
\end{equation}

By the argument of Schatz \cite{S74}, the discrete inf-sup
condition of the sesquilinear form $a_N$ may be established  for sufficiently large $N$ 
and sufficiently small $h$. It follows from the general theory in
\cite{BA73} that the truncated variational problem \eqref{fem} admits a unique
solution. In this work, our goal is to obtain the a posteriori 
error estimate and develop the associated adaptive algorithm. Thus we assume
that the discrete problem \eqref{fem} has a unique solution $u_h^N\in V_h$.

\section{A posteriori error analysis}

First, we collect some relevant results from \cite{KH-SP15} on the Hankel
functions. Let $j_n(t)$ and $y_n(t)$ be the
spherical Bessel functions of the first and second kind with order $n$,
respectively. The spherical Hankel functions are
\[
h_n^{(j)}(t)=j_n(t)\pm{\rm i} y_n(t),\quad j=1,2.
\]
For fixed $t$, the spherical Bessel functions admit the asymptotic
expressions (cf. \cite[Theorem 2.31]{KH-SP15})
\[
 j_n(t)\sim \frac{t^n}{(2n+1)!!},\quad y_n(t)\sim
-\frac{(2n-1)!!}{t^{n+1}},\quad n\to\infty,
\]
which give that 

\begin{equation}\label{asy}
h_n^{(j)}(t)\sim (-1)^j{\rm i}\frac{(2n-1)!!}{t^{n+1}},\quad n\to\infty. 
\end{equation}

For any $K\in \mathcal{M}_h$, let $\mathcal{B}_F$ represent the set of all the
faces of $K$. Denote by $h_K$ and $h_F$ the sizes of element $K$ and face $F$,
respectively. For any interior face $F$ which is the common part of elements
$K_1$ and $K_2$, we define the jump residual across $F$ as
\[
J_F=-(\nabla u_h^N|_{K_1}\cdot \nu_1+\nabla u_h^N|_{K_2}\cdot \nu_2),
\]
where $\nu_j$ is the unit normal vector to the boundary of $K_j, j=1,2$. For any
boundary face $F\in \partial B_R$, we define the jump residual
\[
J_F=2(T u_h^N+\nabla u_h^N\cdot \nu),
\]
 where $\nu$ is the unit outward normal on $\partial B_R$. For any boundary face
$F\in \partial D$, we define the jump residual
\[
J_F=2(\nabla u_h^N\cdot \nu+g),
\]
where $\nu$ is the unit outward normal on $\partial D$ pointing toward
$\Omega$. For any $K\in\mathcal{M}_h$, denote by $\eta_K$ the local
error estimator, which is defined by
\[
\eta_K=h_K\| (\Delta+\kappa^2)
u_h^N\|_{L^2(K)}+\Big(\frac{1}{2}\sum_{F\in\partial
K}h_F\|J_F\|_{L^2(F)}^2\Big)^{\frac{1}{2}}.
\]

We now state the main result, which plays an important role for the numerical
experiments.

\begin{theorem}\label{thm}
Let $u$ and $u_h^N$ be the solutions of \eqref{wf} and \eqref{fem},
respectively. There exists a positive integer $N_0$ independent of $h$ such
that the following a posteriori error estimate holds for $N>N_0$:
\[
\|u-u_h^N\|_{H^1(\Omega)}\lesssim
\bigg(\sum_{K\in\mathcal{M}_h}\eta_K^2\bigg)^{\frac{1}{2}}+
\bigg(\frac{R'}{R}\bigg)^N\|g\|_{L^2(\partial D)}.
\]
\end{theorem}

It can be seen from Theorem \ref{thm} that the a posteriori error consists of
two parts: the first part comes from the finite element discretization error
and the second part accounts for the truncation error of the DtN operator,
which decays exponentially with respect to $N$ since $R'<R$. We point out that
the constant in the estimate may depend on $\kappa, R$ and $R'$, but does not
depend on the truncation parameter of the DtN operator $N$ or the mesh size of
the triangulation $h$.

In the rest part of this section, we prove the a posteriori error estimator in
Theorem \ref{thm} by using a duality argument. 

Denote the error $\xi:=u-u_h^N$. Introduce a dual problem to the original
scattering problem: Find $w\in H^1(\Omega)$ such that 
\begin{equation}\label{bi}
  a(v,w)=(v,\xi)  \quad\forall\, v\in H^{1}(\Omega).
\end{equation}
It is easy to verify that $w$ satisfies the following boundary value
problem:
\begin{equation}\label{dual}
\begin{cases}
\Delta w+ \kappa^2 w= -\xi \quad &\text{in} ~
\mathbb{R}^3\setminus\overline{D},\\
\partial_\nu w = 0 \quad &\text{on} ~\partial D, \\
\partial_r w-T^{*}w=0  \quad &\text{on} ~\partial B_R,
\end{cases}
\end{equation}
where the adjoint operator $T^{*}$ is defined by
\[
(T^*u)(R, \hat x)=\frac{1}{R}\sum_{n=0}^{\infty}\overline{\Theta}_n(\kappa
R)\sum_{m=-n}^{m=n}\hat{u}_n^m Y_n^m(\hat x).
\]
We may follow the same proof as that for the original scattering problem
(\ref{eq}) and show that the dual problem (\ref{dual}) has a unique
weak solution $w\in H^1(\Omega)$, which satisfies
\[
\|w\|_{H^1(\Omega)}\lesssim\|\xi\|_{L^2(\Omega)}.
\]

The following lemma gives the error representation formulas and is the basis for
the a posteriori error analysis.

\begin{lemma}
Let $u$, $u_h^N$ and $w$ be the solutions of the problems $(\rm\ref{wf})$,
$(\rm\ref{fem})$ and $(\rm\ref{bi})$, respectively. The following identities hold: 
\begin{eqnarray}
&&\|\xi\|^2_{H^1(\Omega)}=\Re\left(a(\xi,\xi)+\langle(T-T_N)\xi,\xi
\rangle_{\partial B_R}\right)+\Re\langle T_N\xi,\xi\rangle_{\partial
B_R}+(\kappa^2+1)\|\xi\|^2_{L^2(\Omega)},\label{H1norm}\\
&&\|\xi\|^2_{L^2(\Omega)}=a(\xi, w)+\langle(T-T_N)\xi, w
\rangle_{\partial B_R}-\langle(T-T_N)\xi, w\rangle_{\partial B_R},\label{L2}\\
&&a(\xi,\psi)+\langle(T-T_N)\xi,\psi\rangle_{\partial B_R}=\langle
g,\psi-\psi_h\rangle_{\partial D}-a_N(u_h^N,\psi-\psi_h)\notag\\    
&&\hspace{4cm}+\langle(T-T_N)u,\psi\rangle_{\partial B_R}\quad\forall\, \psi\in
H^1(\Omega), \psi_h\in V_h.\label{p3}
\end{eqnarray}
\end{lemma}

\begin{proof}
The equality (\ref{H1norm}) follows directly from the definition of the
sesquilinear form $a$ in (\ref{wf}). The identity (\ref{L2}) can be easily
deduced by taking $v=\xi$ in (\ref{bi}). It remains to prove (\ref{p3}). It
follows from (\ref{wf}) and (\ref{fem}) that 
\begin{eqnarray*}
a(\xi,\psi)&=&a(u-u_h^N,\psi-\psi_h)+a(u-u_h^N,\psi_h)\\
&=&\langle g,\psi-\psi_h\rangle_{\partial
D}-a(u_h^N,\psi-\psi_h)+a(u-u_h^N,\psi_h)\\
&=&\langle g,\psi-\psi_h\rangle_{\partial D}-a_{N}(u_h^N,\psi-\psi_h)\\
&&+a_N(u_h^N,\psi-\psi_h)-a(u_h^N,\psi-\psi_h)+a(u,\psi_h)-a(u_h^N,
\psi_h).
\end{eqnarray*}
Since $a(u,\psi_h)=\langle g,\psi_h\rangle_{\partial D}=a_N(u_h^N,\psi_h)$, we
have 
\begin{align*}
a(\xi,\psi)&=\langle g,\psi-\psi_h\rangle_{\partial
D}-a_{N}(u_h^N,\psi-\psi_h)+a_N(u_h^N,\psi)-a(u_h^N,\psi)\\
&=\langle g,\psi-\psi_h\rangle_{\partial D}-a_{N}(u_h^N,\psi-\psi_h)+\langle
(T-T_N)u_h^N,\psi\rangle_{\partial B_R}\\
&=\langle g, \psi-\psi_h\rangle_{\partial D}-a_{N}(u_h^N,\psi-\psi_h)-\langle
(T-T_N)\xi,\psi\rangle_{\partial B_R}\\
&\quad +\langle (T-T_N)u,\psi\rangle_{\partial B_R},
\end{align*}
which implies (\ref{p3}) and completes the proof. 
\end{proof}

It is necessary to estimate \eqref{p3} and the last term in \eqref{L2} in
order to prove Theorem \ref{thm}. We begin with a trace regularity result.

\begin{lemma}\label{trace}
For any $u\in H^1(\Omega)$, the following estimates hold: 
\[
 \|u\|_{H^{\frac{1}{2}}(\partial B_R)}\lesssim \|u\|_{H^1(\Omega)},\quad
 \|u\|_{H^{\frac{1}{2}}(\partial B_{R'})}\lesssim \|u\|_{H^1(\Omega)}.
\]
\end{lemma}

\begin{proof}
Let
\[
 B_R\setminus\overline{B}_{R'}=\{(r,\theta,\varphi):0<R'<r<R,~0<\theta<\pi,
~0<\varphi<2\pi\}.
\]
It is shown in \cite[Lemma 2]{JLLZ-JSC17} that 
\[
 (1+n^2)^{\frac{1}{2}}|\zeta(R)|^2
 \lesssim\int_{R'}^{R}\left((1+n^2)|\zeta(r)|^2+|\zeta'(r)|^2\right){\rm d}r,
\]
which gives after combining (\ref{seminorm}) that 
\begin{eqnarray*}
\|u\|^2_{H^{\frac{1}{2}}(\partial
B_R)}&=&\sum_{n=0}^{\infty}(1+n^2)^{\frac{1}{2}
}\sum\limits_{m=-n}^{n}|\hat{u}_n^m (R)|^2\\
&\lesssim&\sum_{n=0}^{\infty} \sum_{|m|\leq
n}\int_{R'}^{R}\big((1+n^2)|\hat{u}_n^m(r)|^2+|\hat{u}_n^{m'}(r)|^2\big){\rm
d}r.
\end{eqnarray*}
Noting the fact (cf. \cite[Theorem 5.34]{KH-SP15})
\[
\|u\|^2_{H^1({B}_R\setminus \overline{B}_{R'})}\geq\sum_{n=0}^{\infty}\sum_{
|m|\leq n}\int_{R'}^{R}r^2\Big[\Big(1+\frac{n(n+1)}{r^2}\Big)|\hat{u}_n^m(r)|^2
+|\hat{u}_n^{m '}(r)|^2\Big] {\rm d}r,
\]
we obtain
\[
\|u\|^2_{H^{\frac{1}{2}}(\partial B_R)}\lesssim\|u\|^2_{H^1(B_R\setminus
\overline{B}_{R'})}\leq \|u\|^2_{H^1(\Omega)},
\]
which shows the first inequality. The second inequality can be proved similarly
by observing the identity
\[
(R-R')|\zeta(R')|^2=\int_{R'}^R|\zeta(r)|^2{\rm d}r
+\int_{R'}^R\int_r^{R'}\frac{{\rm d}}{{\rm d}r}|\zeta(r)|^2{\rm
d}t{\rm d}r.
\]
The details are omitted here. 
\end{proof}

\begin{lemma}\label{lem-uu}
Let $u$ be the solution to $\rm(\ref{wf})$. Then the following estimate holds:
\[
|\hat{u}_n^m(R)|\lesssim \Big(\frac{R'}{R}\Big)^{n}|\hat{u}_n^m (R')|.
\]
\end{lemma}

\begin{proof}
It is known that the solution of the scattering problem
\eqref{eq} admits the series expansion
\begin{equation}\label{ur}
u(r, \hat{x})=\sum\limits_{n=0}^{\infty}\sum\limits_{m=-n}^n\frac{h_n^{(1)}(\kappa
r)}{h_n^{(1)}(\kappa R')}\hat{u}_n^m(R')Y_n^m(\hat{x}),\quad
\hat{u}_n^m(R')=\int_{\mathbb{S}^2}u(R', \hat{x})Y_n^m(\hat{x}){\rm
d}\hat{x}
\end{equation}
for all $r>R'$. Evaluating \eqref{ur} at $r=R$ yields
\[
u(R, \hat{x})=\sum\limits_{n=0}^{\infty}\sum\limits_{m=-n}^n\frac{h_n^{(1)}(\kappa
R)}{h_n^{(1)}(\kappa R')}\hat{u}_n^m(R')Y_n^m(\hat{x}),
\]
which implies
\[
\hat{u}_n^m(R)=\frac{h_n^{(1)}(kR)}{h_n^{(1)}(kR')}\hat{u}_n^m(R').
\]
Using the asymptotic expression in \eqref{asy}, we obtain
\[
|\hat{u}_n^m(R)|=\Biggl|\frac{h_n^{(1)}(kR)}{h_n^{(1)}(kR')}\Biggr|
|\hat{u}_n^m(R')|\lesssim\Big(\frac{R'}{R}\Big)^{n}|\hat{u}_n^m(R')|,
\]
which completes the proof.
\end{proof}

\begin{lemma}\label{lem-est}
For any $\psi\in H^1(\Omega)$, the following estimate holds:
\[
|a(\xi,\psi)+\langle(T-T_N)\xi,\psi\rangle_{\partial B_R}
|\lesssim \bigg(\Big(\sum_{K\in\mathcal{M}_h}\eta_K^2\Big)^{\frac{1}{2}}+
\Big(\frac{R'}{R}\Big)^N\|g\|_{L^2(\partial
D)}\bigg)\|\psi\|_{H^1(\Omega)}.
\]
\end{lemma}

\begin{proof}
Define
\begin{eqnarray*}
J_1&=&\langle g,\psi-\psi_h\rangle_{\partial D}-a_N(u_h^N,\psi-\psi_h),\\
J_2&=&\langle (T-T_N)u,\psi\rangle_{\partial B_R},
\end{eqnarray*}
where $\psi_h\in V_h$. It follows from (\ref{p3}) that
\[
a(\xi,\psi)+\langle(T-T_N)\xi,\psi\rangle_{\partial
B_R}=J_1+J_2.
\]
Using (\ref{febi}) and the integration by parts, we obtain
\[
J_1=\sum_{K\in\mathcal{M}_h}\biggl(\int_K (\Delta u_h^N+\kappa^2 u_h^N)
(\bar{\psi}-\bar{\psi}_h){\rm d}x+\sum_{F\in\partial
K}\frac{1}{2}\int_{F}J_F(\bar{\psi}-\bar{\psi}_h){\rm d}s\biggl).
\]
Now we take $\psi_h=\Pi_h\psi\in V_h$, where $\Pi_h$ is the Scott--Zhang
interpolation operator and has the approximation properties
\[
\|v-\Pi_h v\|_{L^2(K)}\lesssim h_K\|\nabla v\|_{L^2(\tilde{K})},\quad \|v-\Pi_h
v\|_{L^2(F)}\lesssim h_F^{\frac{1}{2}}\|\nabla v\|_{L^2(\tilde{K}_F)},
\]
Here $\tilde{K}$ and $\tilde{K}_F$ are the union of all the elements
in $\mathcal{M}_h$, which have nonempty intersection with element $K$
and the face $F$, respectively.

By the Cauchy--Schwarz inequality, we have 
\begin{eqnarray*}
 |J_1|&\lesssim& \sum\limits_{K\in\mathcal{M}_h}
 \bigg(h_K\|(\Delta+\kappa^2)u_h^N\|_{L^2(K)}\|\nabla\psi\|_{L^2(\tilde{K})}
 +\sum\limits_{F\in\partial K}\frac{1}{2}h_F^{\frac{1}{2}}\|J_F\|_{L^2(F)}
 \|\psi\|_{H^1(\tilde{K}_F)}\bigg)\\
&\lesssim&\sum_{K\in\mathcal{M}_h}\bigg[h_K\|(\Delta+\kappa^2)u_h^N\|_{L^2(K)}
+\bigg(\sum_{F\in\partial
K}\frac{1}{2}h_F\|J_F\|_{L^2(F)}^2\bigg)^{\frac{1}{2}}\bigg]
\|\psi\|_{H^1(\Omega)}\\
 &\lesssim&
\bigg(\sum_{K\in\mathcal{M}_h}\eta_K^2\bigg)^2\|\psi\|_{H^1(\Omega)}.
\end{eqnarray*}
It follows from the definitions of $T, T_N$ and Lemma \ref{lem-uu} that
\begin{eqnarray*}
|J_2|=|\langle(T-T_N)u,\psi\rangle_{\partial B_R}|
&=&\Big{|}R\sum_{n>N}\sum_{|m|\leq n}\Theta_n(\kappa R)\hat{u}_n^m(R)
\bar{\hat{\psi}}_n^m(R)\Big{|}\\
&\lesssim& \sum_{n>N}\sum_{|m|\leq n}|\Theta_n(\kappa
R)||\hat{u}_n^m(R)||\bar{\hat{\psi}}_n^m(R)|\\
&\lesssim& \sum_{n>N}\sum_{|m|\leq n}|\Theta_n(\kappa R)
|\Big{|}\Big(\frac{R'}{R}\Big)^N\hat{u}_n^m(R')\Big{|}
|\bar{\hat{\psi}}_{nm}(R)|\\
&\lesssim& \Big(\frac{R'}{R}\Big)^N\sum_{n>N}\sum_{|m|\leq
n}|\Theta_n(\kappa R)||\hat{u}_n^m(R')||\bar{\hat{\psi}}_n^m(R)|.
\end{eqnarray*}
Using (\ref{pro}), the Cauchy--Schwarz inequality, and Lemma \ref{trace} yields
\begin{eqnarray*}
|J_2| &\lesssim& \Big(\frac{R'}{R}\Big)^N\sum_{n>N}\sum_{|m|\leq
n}|\Theta_n(\kappa R)||\hat{u}_n^m(R')||\bar{\hat{\psi}}_n^m(R)|\\
&\lesssim& \Big(\frac{R'}{R}\Big)^N\sum_{n>N}(1+n(n+1))^{\frac{1}{2}}\sum_{|m|
\leq n}|\hat{u}_n^m(R')||\bar{\hat{\psi}}_n^m(R)|\\
&\leq&\Big(\frac{R'}{R}\Big)^N\sum_{n>N}(1+n(n+1))^{\frac{1}{2}}\Big(\sum_{
|m|\leq n}|\hat{u}_n^m(R')|^2\Big)^{\frac{1}{2}}\Big(\sum_{|m|\leq n}|
\bar{\hat{\psi}}_n^m(R)|^2\Big)^{\frac{1}{2}}\\
&\leq&
\Big(\frac{R'}{R}\Big)^N\Big(\sum_{n>N}(1+n(n+1))^{\frac{1}{2}}\sum_{|m|\leq
n}|\hat{u}_n^m(R')|^2\Big)^{\frac{1}{2}}\Big(\sum_{n>N}(1+n(n+1))^{\frac{1}{2}}
\sum_{|m|\leq n}|\bar{\hat{\psi}}_n^m(R)|^2\Big)^{\frac{1}{2}}\\
&\lesssim& \Big(\frac{R'}{R}\Big)^N\|u\|_{H^{\frac{1}{2}}(\partial
B_{R'})}\|\psi\|_{H^{\frac{1}{2}}(\partial B_R)}\\
&\lesssim&\Big(\frac{R'}{R}\Big)^N\|u\|_{H^{1}(\Omega)}\|\psi\|_{H^{\frac{1}{2}}
(\partial B_R)}.
\end{eqnarray*}
Combining the above estimates gives
\[
|J_1|+|J_2|\lesssim
\bigg(\Big(\sum_{K\in\mathcal{M}_h}\eta_K^2\Big)^{\frac{1}{2}}+
\Big(\frac{R'}{R}\Big)^N\|g\|_{L^2(\partial D)}\bigg)\|\psi\|_{H^1(\Omega)},
\]
which completes the proof.
\end{proof}

\begin{lemma}
Let $w$ be the solution to the dual problem $\rm(\ref{bi})$. Then the following
estimate holds: 
\[
|\langle (T-T_N)\xi,w\rangle_{\partial
B_R}|\lesssim N^{-2}\|\xi\|^2_{H^1(\Omega)}.
\]
\end{lemma}

\begin{proof}
It  follows from (\ref{pro}), Lemma \ref{trace} and the Cauchy--Schwarz
inequality that we have
\begin{eqnarray*}
|\langle (T-T_N)\xi,w\rangle_{\partial B_R}|
&\lesssim& \sum_{n>N}\sum_{|m|\leq n}|\Theta_n(\kappa R)
||\hat{\xi}_n^m(R)||\hat{w}_n^m(R)|\\
&\lesssim& \sum_{n>N}\sum_{|m|\leq
n}|n||\hat{\xi}_n^m(R)||\hat{w}_n^m(R)|\\
&=&\sum_{n>N} ((1+n^2)^{\frac{1}{2}}n^3)^{-\frac12}\sum_{|m|\leq n}
(1+n^2)^{\frac14}n^{\frac52}|\hat{\xi}_n^m(R)||\hat{w}_n^m(R)|\\
&\leq& N^{-2}\sum_{n>N} \sum_{|m|\leq n}(1+n^2)^{\frac14}n^{\frac52}
|\hat{\xi}_n^m(R)||\hat{w}_n^m(R)|\\
&\leq& N^{-2}\bigg(\sum_{n>N} \sum_{|m|\leq n}(1+n(n+1))^{\frac12}
|\hat{\xi}_n^m(R)|^2\bigg)^{\frac12}\bigg(\sum_{n>N}
\sum_{|m|\leq n}n^{5}|\hat{w}_n^m(R)|^2\bigg)^{\frac12}\\
&=&N^{-2}\|\xi\|_{H^{\frac12}(\partial B_R)}\bigg(\sum_{n>N} \sum_{|m|\leq
n}n^{5}|\hat{w}_n^m(R)|^2\bigg)^{\frac12}\\
& \lesssim & N^{-2}\|\xi\|_{H^{1}(\Omega)}\bigg(\sum_{n>N} \sum_{|m|\leq
n}n^{5}|\hat{w}_n^m(R)|^2\bigg)^{\frac12}.
\end{eqnarray*}
To estimate $\hat{w}_n^m(R)$, we consider the dual problem
(\ref{dual}) in the annulus $B_R\setminus B_{R'}$: 
\[
\begin{cases}
\Delta w + \kappa^2 w =-\xi \quad &\text{in} ~ B_R\setminus
\overline{B}_{R'}, \\
w = w(R',\hat{x}) \quad &\text{on} ~\partial R', \\
\partial_r w-T^* w=0  \quad &\text{on} ~\partial B_R,
\end{cases}
\]
which reduces to the second order equation for the coefficients
$\hat{w}_n^m$ in the Fourier domain
\[
\begin{cases}
\frac{{\rm d}^2\hat{w}_n^m(r)}{{\rm d}r^2}+\frac{2}{r}\frac{{\rm
d}\hat{w}_n^m(r)}{{\rm d}r}+\big(\kappa^2-\frac{n(n+1)}{r^2}\big)\hat{
w}_n^m(r)=-\hat{\xi}_n^m(r), & R'<r<R, \\
\frac{{\rm d}\hat{w}_n^m(R)}{{\rm
d}r}-\frac{1}{R}\overline{\Theta}_n(\kappa R)\hat{w}_n^m(R)=0, & r=R,\\
\hat{w}_n^m(R')=\hat{w}_n^m(R'), & r=R'.
\end{cases}
\]
By the method of the variation of parameters, we obtain the solution of the
above equation
\begin{equation}\label{wnm}
\hat{w}_n^m(r)=S_n(r)\hat{w}_n^m(R')+\frac{{\rm i
\kappa}}{2}\int_{R'}^{r}t^2W_n(r,t)\hat{\xi}_n^m(t){\rm d}t
+\frac{{\rm i\kappa}}{2}\int_{R'}^{R}t^2S_n(t)W_n(R',r)\hat{\xi}_n^m(t){\rm
d}t,
\end{equation}
where
\[
S_n(r)=\frac{h_n^{(2)}(\kappa r)}{h_n^{(2)}(\kappa R')},\quad 
W_n(r,t)=\det\left[
\begin{matrix}
h_n^{(1)}(\kappa r) & h_n^{(2)}(\kappa r)  \\
h_n^{(1)}(\kappa t) & h_n^{(2)}(\kappa t)
\end{matrix}
\right] .
\]
Taking $r=R$ in (\ref{wnm}), we get
\[
\hat{w}_n^m(R)=S_n(R)\hat{w}_n^m(R')+\frac{{\rm
i}\kappa}{2}\int_{R'}^{R}t^2S_n(R)W_n(R',t)\hat{\xi}_n^m(t){\rm d}t.
\]
Using the asymptotic expression (\ref{asy}) yields 
\[
S_n(R)\sim\Big(\frac{R'}{R}\Big)^n, \quad n\rightarrow \infty
\]
and
\begin{eqnarray*}
W_n(R',t)&=&2{\rm i}j_n(\kappa R')y_n(\kappa R')\bigg(\frac{j_n(\kappa
t)}{j_n(\kappa R')}-\frac{y_n(\kappa t)}{y_n(\kappa R')}\bigg)\\
&\sim&-\frac{2{\rm i}}{(2n+1)\kappa
R'}\bigg(\Big(\frac{t}{R'}\Big)^n-\Big(\frac{R'}{t}\Big)^{n+1}\bigg), 
\quad n\rightarrow \infty.
\end{eqnarray*}
Hence
\[
|S_n(R)|\lesssim\Big(\frac{R'}{R}\Big)^n, \quad | W_n(R',t)|\lesssim
n^{-1}\Big(\frac{t}{R'}\Big)^{n}.
\]
Combining the above estimates, we obtain 
\begin{eqnarray*}
|\hat{w}_n^m(R)|&\leq& |S_n(R)||\hat{w}_n^m(R')|
+\frac{\kappa}{2}\int_{R'}^{R}t^2|S_n(R)||W_n(R',t)||\hat{\xi}_n^m(t)|{\rm
d}t,\\
&\lesssim&\Big(\frac{R'}{R}\Big)^n|\hat{w}_n^m(R')|+
n^{-1}\Big(\frac{R'}{R}\Big)^{n}\|\hat{\xi}_n^m(t)\|_{L^\infty([R',R])}
\int_{R'}^{R}t^2\Big(\frac{t}{R'}\Big)^{n}{\rm d}t\\
&\lesssim&\Big(\frac{R'}{R}\Big)^n|\hat{w}_n^m(R')|+
n^{-2}\|\hat{\xi}_n^m(t)\|_{L^\infty([R',R])},
\end{eqnarray*}
which gives
\begin{eqnarray*}
\sum_{n>N} \sum_{|m|\leq n}n^{5}|\hat{w}_n^m(R)|^2
&\lesssim& \sum_{n>N} \sum_{|m|\leq n}n^{5}\bigg(\Big(\frac{R'}{R}\Big)^n
|\hat{w}_n^m(R')|+n^{-2}\|\hat{\xi}_n^m(t)\|_{L^\infty([R',R])}
\bigg)^2\\
&\lesssim& \sum_{n>N} \sum_{|m|\leq
n}n^{5}\bigg(\Big(\frac{R'}{R}\Big)^{2n}
|\hat{w}_n^m(R')|^2+n^{-4}\|\hat{\xi}_n^m(t)\|^2_{L^\infty([R',R])}
\bigg)\\
 &:=&I_1+I_2.
\end{eqnarray*}
Here
\begin{align*}
 I_1=&\sum_{n>N} \sum_{|m|\leq n}n^{5}\Big(\frac{R'}{R}\Big)^{2n}
     |\hat{w}_n^m(R')|^2,\\
 I_2=&\sum_{n>N} \sum_{|m|\leq n}n\|\hat{\xi}_n^m(t)\|^2_{L^\infty([R',R])}.
\end{align*}
A simple calculation yields
\begin{eqnarray*}
I_1&\lesssim &\max_{n>N}n^{4}\Big(\frac{R'}{R}\Big)^{2n}\sum_{n>N}
\sum_{|m|\leq n}n~|\hat{w}_n^m(R')|^2\lesssim \sum_{n>N} \sum_{|m|\leq
n}n~|\hat{w}_n^m(R')|^2\\
&\lesssim &\sum_{n>N} (1+n^2)^{\frac12}\sum_{|m|\leq
n}|\hat{w}_n^m(R')|^2\leq\|w\|^2_{H^{\frac12} (\partial B_{R'})}
\lesssim \|\xi\|^2_{H^1(\Omega)}.
\end{eqnarray*}
By \cite[Lemma 5]{JLLZ-JSC17}, we have
\[
 \|\hat{\xi}_n^m(t)\|^2_{L^\infty([R',R])}\leq
 \Big(\frac{2}{\delta}+n\Big)\|\hat{\xi}_n^m(t)\|^2_{L^2([R',R])}
 +n^{-1}\|\hat{\xi}_n^{m'}(t)\|^2_{L^2([R',R])},
\]
where $\delta=R-R'$. Following a similar proof of Lemma 4.2 yields
\[
\begin{split}
\|\xi\|^2_{H^1(B_R\setminus \overline{B}_R)}&\geq\sum_{n=0}^{\infty}
\sum_{|m|\leq n}\int_{R'}^{R}\left[(r^2+n(n+1))|\xi_n^m(r)|^2+r^2|\xi_n^{m'}
(r)|^2\right] {\rm d}r\\
&\geq \sum_{n=0}^{\infty}\sum_{|m|\leq n}\int_{R'}^{R}\left[({R'}
^2+n(n+1))|\xi_n^m(r)|^2+{R'}^2|\xi_n^{m'}(r)|^2\right ] {\rm d}r,
\end{split}
\]
which gives 
\[
I_2=\sum_{n>N} \sum_{|m|\leq n}n\|\hat{\xi}_{nm}(t)\|^2_{L^\infty([R',R])}
\lesssim\|\xi\|^2_{H^1(\Omega)}.
\]
Therefore, we obtain 
\[
 \sum_{n>N} \sum_{|m|\leq
n}n^{5}|\hat{w}_n^m(R)|^2\lesssim\|\xi\|^2_{H^1(\Omega)},
\]
which completes the proof.
\end{proof}

Now we prove the main theorem. 

\begin{proof}
We conclude from (\ref{tu})--(\ref{pro}) that 
\[
 \Re\langle T_N\xi,\xi\rangle_{\partial B_R}=R\sum_{n>N}
 \sum_{|m|\leq n}\Re(\Theta(\kappa R))|\hat{\xi}_n^m|^2\leq 0.
\]
It follows from (\ref{H1norm}) and Lemma \ref{lem-uu} that there exist two positive constants $C_1$ and $C_2$
independent of $h$ and $N$ satisfying
\[
 \|\xi\|^2_{H^1(\Omega)}\leq
C_1\bigg(\Big(\sum_{K\in\mathcal{M}_h}\eta_K^2\Big)^{\frac12}+
     \Big(\frac{R'}{R}\Big)^N\|g\|_{L^2(\partial D)}\bigg)
      \|\xi\|_{H^1(\Omega)}+C_2 \|\xi\|_{L^2(\Omega)}.
\]
Using (\ref{L2}) and Lemmas \ref{lem-uu}--\ref{lem-est}, we obtain
\[
 \|\xi\|^2_{L^2(\Omega)}\leq
C_3\bigg(\Big(\sum_{K\in\mathcal{M}_h}\eta_K^2\Big)^{\frac12}+
     \Big(\frac{R'}{R}\Big)^N\|g\|_{L^2(\partial D)}\bigg)
     \|\xi\|_{L^2(\Omega)}+C_4N^{-2} \|\xi\|_{H^1(\Omega)},
\]
where $C_3$ and $C_4$ are positive constants independent of $h$ and $N$.
Combining the above estimates yields
\[
 \|\xi\|^2_{H^1(\Omega)}\leq
C_5\bigg(\Big(\sum_{K\in\mathcal{M}_h}\eta_K^2\Big)^{\frac12}+
     \Big(\frac{R'}{R}\Big)^N\|g\|_{L^2(\partial D)}\bigg)
     \|\xi\|_{H^1(\Omega)}+C_6N^{-2} \|\xi\|_{H^1(\Omega)},
\]
where $C_5$ and $C_6$ are positive constants independent of $h$ and $N$.
We may choose a sufficiently large integer $N_0$ such that $C_6N_0^{-2}<1/2$,
which completes the proof by taking $N>N_0$. 
\end{proof}

\section{Numerical experiments}

In this section, we discuss the implementation of the adaptive finite element
algorithm with the truncated DtN boundary condition and present two numerical
examples to demonstrate the competitive performance of the proposed method.
There are two components which need to be designed carefully in order to
efficiently implement the h-adaptive method. The first one is an effective
management mechanism of the mesh grids. The another one is an effective
indicator for the adaptivity. The a posteriori error estimate from Theorem
\ref{thm} is used to generate the indicator in our algorithm.

\subsection{The hierarchy geometry tree}

In our algorithm, we use the hierarchy geometry tree (HGT) or the hierarchical
grids to manage the data structure of the mesh grids \cite{BHL12}. The
structure of the grids is described hierarchically. For example, the element
such as a point for 0-dimension, an edge for 1-dimension, a triangle for
2-dimension, a tetrahedron for 3-dimension is called a geometry. If a
triangle is one of the faces of a tetrahedron, then it belongs to this
tetrahedron. Similarly, if an edge is one of the edges of a triangle, then it
belongs to this triangle. Hence all geometries in the tetrahedrons have
belonging-to relationship.
 
A tetrahedron $T_{0}$ can be uniformly divided into eight small sub-tetrahedrons
$\{T_{0,0},T_{0,1} \cdots T_{0,7} \}$.  In this refinement operation, every
face of the tetrahedron is divided into four smaller triangles. This procedure
can be managed by the octree data structure which is given by Figure
\ref{fig1}, which shows that the sub-tetrahedrons $T_{0,0}$ and
$T_{0,6}$ are further divided into eight smaller sub-tetrahedrons. In the 
octree, we name $T_0$ as the root node and those nodes without further
subdivision like $T_{0,1}$ and $T_{0,0,0}$ as the leaf nodes. Obviously, 
a set of root nodes $\{T_{i}\}, i =0,1,\cdots$ can form a three-dimensional
initial mesh for a domain $\Omega$ and a set of all the leaf nodes of the HGTs
also form a mesh. 

\begin{figure}[h]
\centering
\begin{tikzpicture}[scale = 0.8]
            
\draw (0,0) node(0) [above=5]{${T_0}$}circle(8pt) ;
\draw (-3.5,-2.5) node(1) [above=5]{$T_{0,0}$}circle(8pt) ;
\draw (-2.5,-2.5) node(2) [above=5]{$T_{0,1}$}circle(8pt) ;
\draw (-1.5,-2.5) node(3) [above=5]{$T_{0,2}$}circle(8pt) ;
\draw (-0.5,-2.5) node(4) [above=5]{$T_{0,3}$}circle(8pt) ;
\draw (0.5,-2.5) node(5) [above=5]{$T_{0,4}$}circle(8pt) ;
\draw (1.5,-2.5) node(6) [above=5]{$T_{0,5}$}circle(8pt) ;
\draw (2.5,-2.5) node(7) [above=5]{$T_{0,6}$}circle(8pt) ;
\draw (3.5,-2.5) node(8) [above=5]{$T_{0,7}$}circle(8pt) ;
\draw [->](0,0) -- (-3.5,-2.5);
\draw [->](0,0) -- (-2.5,-2.5);
\draw [->](0,0) -- (-1.5,-2.5);
\draw [->](0,0) -- (-0.5,-2.5);
\draw [->](0,0) -- (0.5,-2.5);
\draw [->](0,0) -- (1.5,-2.5);
\draw [->](0,0) -- (2.5,-2.5);
\draw [->](0,0) -- (3.5,-2.5);
\draw (-8,-5) node(9) [above=5]{$T_{0,0,0}$}circle(8pt) ;
\draw (-7.,-5) node(10) [above=5]{$T_{0,0,1}$}circle(8pt) ;
\draw (-6,-5) node(11) [above=5]{$T_{0,0,2}$}circle(8pt) ;
\draw (-5,-5) node(12) [above=5]{$T_{0,0,3}$}circle(8pt) ;
\draw (-4,-5) node(13) [above=5]{$T_{0,0,4}$}circle(8pt) ;
\draw (-3,-5) node(14) [above=5]{$T_{0,0,5}$}circle(8pt) ;
\draw (-2,-5) node(15) [above=5]{$T_{0,0,6}$}circle(8pt) ;
\draw (-1,-5) node(16) [above=5]{$T_{0,0,7}$}circle(8pt) ;
\draw (1.5,-5) node(9) [above=5]{$T_{0,6,0}$}circle(8pt) ;
\draw (2.5,-5) node(10) [above=5]{$T_{0,6,1}$}circle(8pt) ;
\draw (3.5,-5) node(11) [above=5]{$T_{0,6,2}$}circle(8pt) ;
\draw (4.5,-5) node(12) [above=5]{$T_{0,6,3}$}circle(8pt) ;
\draw (5.5,-5) node(13) [above=5]{$T_{0,6,4}$}circle(8pt) ;
\draw (6.5,-5) node(14) [above=5]{$T_{0,6,5}$}circle(8pt) ;
\draw (7.5,-5) node(15) [above=5]{$T_{0,6,6}$}circle(8pt) ;
\draw (8.5,-5) node(16) [above=5]{$T_{0,6,7}$}circle(8pt) ;
\draw [->](-3.5,-2.5) -- (-8,-5);
\draw [->](-3.5,-2.5) -- (-7,-5);
\draw [->](-3.5,-2.5) -- (-6,-5);
\draw [->](-3.5,-2.5) -- (-5,-5);
\draw [->](-3.5,-2.5) -- (-4,-5);
\draw [->](-3.5,-2.5) -- (-3,-5);
\draw [->](-3.5,-2.5) -- (-2,-5);
\draw [->](-3.5,-2.5) -- (-1,-5);
\draw [->](2.5,-2.5) -- (1.5,-5);
\draw [->](2.5,-2.5) -- (2.5,-5);
\draw [->](2.5,-2.5) -- (3.5,-5);
\draw [->](2.5,-2.5) -- (4.5,-5);
\draw [->](2.5,-2.5) -- (5.5,-5);
\draw [->](2.5,-2.5) -- (6.5,-5);
\draw [->](2.5,-2.5) -- (7.5,-5);
\draw [->](2.5,-2.5) -- (8.5,-5);
\end{tikzpicture}
\caption{A schematic of octree data structure.}
\label{fig1}
\end{figure}
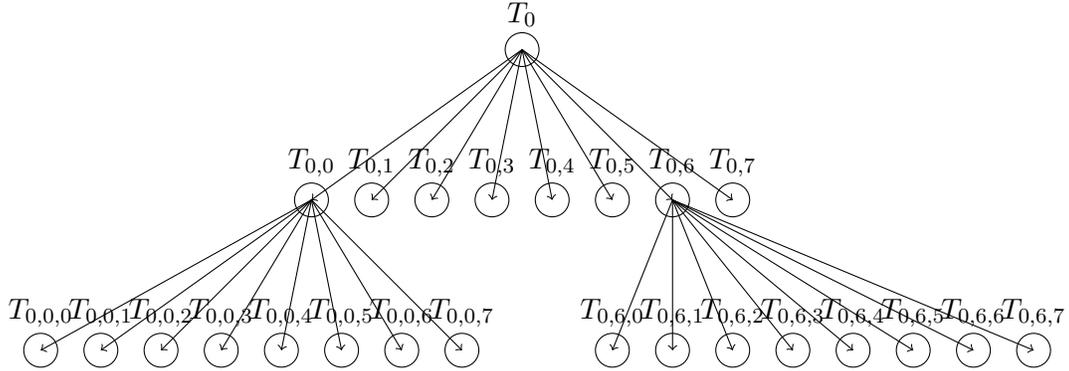

By using the HGT, the refinement and even the coarsening of a mesh can be done
efficiently. However, it may cause the hanging points in the direct neighbors
of the refined tetrahedrons. In order to remove these hanging points, two kinds
of geometries may be introduced: twin-tetrahedron and four-tetrahedron. For the
twin-tetrahedron geometry as shown in Figure \ref{fig2} (left), it has five
degrees of freedom (DoF) and consists of two standard tetrahedrons. To
conform the finite element space, the following strategy is used to construct
the basis function in twin-tetrahedron geometry. For each basis function, the
value is 1 at the corresponding interpolation point and the value is 0 at the
other interpolation points. For the common point of the two sub-tetrahedrons in
the twin-tetrahedron like A, D and E, the support of the basis function is the
whole twin-tetrahedron. For the points B and C, the support of their
corresponding basis function is only the tetrahedron ABED and the tetrahedron
AECD, respectively. For the four-tetrahedron as shown in Figure \ref{fig2}
(right), the similar strategy is used. With the twin-tetrahedron geometry and
the four-tetrahedron geometry, the local refinement can be implemented easily.

\begin{figure}[h]
\centering
\begin{minipage}{0.48\linewidth}
\centering
\begin{tikzpicture}
\draw [fill](0,0) node(B)[left]{B} circle(1pt) ;
\draw [fill](1.6,1) node(D)[left]{D} circle(1pt);
\draw [fill](2,0) node(E) [below]{E} circle(1pt);
\draw [fill](4,0) node(C)[right]{C} circle(1pt);
\draw [fill](1.6,3) node(A) [above]{A} circle(1pt);
\draw[dashed] (0,0)  -- (1.6,1)  -- (4,0);
\draw[dashed] (1.6,3) -- (1.6,1) -- (2,0);
\draw (0,0) -- (4,0);
\draw (0,0) -- (1.6,3);
\draw (2,0) -- (1.6,3);
\draw (1.6,3) -- (4,0);
\end{tikzpicture}
\end{minipage}
\begin{minipage}{0.48\linewidth}
\centering
\begin{tikzpicture}
\draw [fill](0,0) node(B)[left]{B} circle(1pt) ;
\draw [fill](1.6,1) node(D)[left]{D} circle(1pt);
\draw [fill](2,0) node(E) [below]{E} circle(1pt);
\draw [fill](4,0) node(C)[right]{C} circle(1pt);
\draw [fill](1.6,3) node(A) [above]{A} circle(1pt);
\draw [fill](0.8,0.5) node(G) [left]{G} circle(1pt);
\draw [fill](2.8,0.5) node(F) [right]{F} circle(1pt);
\draw[dashed] (0,0)  -- (1.6,1)  -- (4,0);
\draw[dashed] (1.6,3) -- (1.6,1) -- (2,0);
\draw[dashed] (1.6,3) -- (0.8,0.5) --(2,0) --(2.8,0.5) --(0.8,0.5);
\draw[dashed] (1.6,3) -- (2.8,0.5);
\draw (0,0) -- (4,0);
\draw (0,0) -- (1.6,3);
\draw (2,0) -- (1.6,3);
\draw (1.6,3) -- (4,0);
\end{tikzpicture}
\end{minipage}
\caption{Two geometries to avoid hanging points. (left) Twin-tetrahedron
geometry. (right) Four-tetrahedron geometry.}
\label{fig2}
\end{figure}
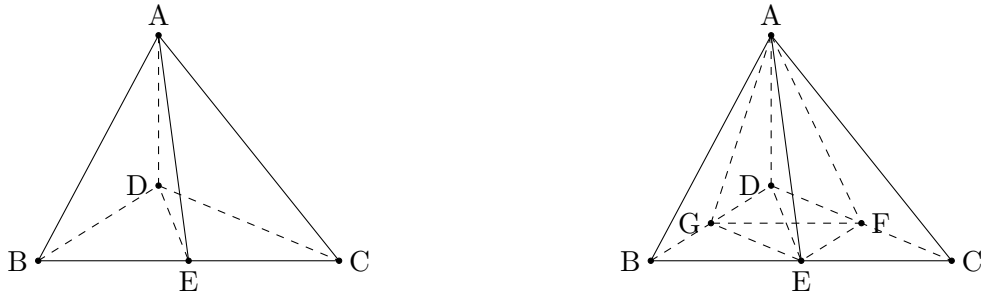

In practice, we use a polyhedral surface to approximate $\partial D$ and
$\partial B_R$. Since the TBC operator is represented by the spherical harmonic
functions whose accuracy depends on how good the approximation is. Obviously, a
rough approximation could not satisfy the computational requirement. Based on
the element geometry introduced above, we present a method to deal with
the surface refinement. Suppose that the domain $\Omega$ has a curved boundary
and the initial mesh is given
by a rough polygon. The traditional surface refinement is performed by taking
the midpoint of each side of the tetrahedron. Hence the shape of the boundary
cannot be well approximated. To resolve this issue, a very simple method is
adopted. When the boundary elements of the mesh need to be refined, we redefine
the midpoint through projecting vertically to the desired curved boundary, as
shown in Figure \ref{fig3}. Thanks to the HGTs, it does not spend much time at
all to find these boundary elements. This method works efficiently in
two-dimensions. But in three-dimensions, it may cause the neighbors to become
non-standard twin-tetrahedron geometry or four-tetrahedrons geometry. To handle
this problem, these special tetrahedrons, whose neighbors do not need
refinement, should not redefine the midpoint. So the marked boundary tetrahedron
will not be refined until the neighboring boundary tetrahedrons are marked in
order to keep the mesh structure. 

\begin{figure}[htp]
\centering
\begin{tikzpicture} 

\filldraw[ball color= white] 
(-1,-1) .. controls (1,-1)  .. (2.5,0) --
(2.5,0) arc (0:90:2.5) --
(0,2.5) ..controls  (-1.2,1) ..(-1,-1);

\draw [<->](-2,-2) node[below] {$x_1$} -- (0,0) -- (4,0) node[right] {$x_2$};
\draw[->] (0,0) -- (0,4) node[above] {$x_3$};

\filldraw (0,0)circle (2pt) ; 
\filldraw (-1,-1)circle (2pt) ; 
\filldraw (2.5,0)circle (2pt);
\filldraw (0,2.5)circle (2pt);
\filldraw (0.75,-0.5)circle (1pt) ; 
\filldraw (-0.5,0.75)circle (1pt);
\filldraw (1.25,1.25)circle (1pt);

\draw [dashed] (0,2.5) -- (-1,-1);
\draw [dashed] (2.5,0) -- (-1,-1);
\draw [dashed] (0,2.5) -- (2.5,0);

\filldraw[color = red] (1.77,1.77)circle (2pt) ; 
\filldraw[color = red] (-0.9,1.3)circle (2pt);
\filldraw[color = red] (1.2,-0.8)circle (2pt);

\draw [dashed] (1.25,1.25) -- (1.77,1.77);
\draw [dashed] (0.75,-0.5) -- (1.2,-0.8);
\draw [dashed] (-0.5,0.75)-- (-0.9,1.3);
\draw [color = red] (-1,-1) -- (-0.9,1.3) -- (0,2.5);
\draw [color = red] (2.5,0) -- (1.77,1.77) -- (0,2.5);
\draw [color = red] (-1,-1) -- (1.2,-0.8) -- (2.5,0);

\end{tikzpicture}
\caption{Mesh refinement on the surface (red points are
redefined midpoints on the boundary).}
\label{fig3}
\end{figure}
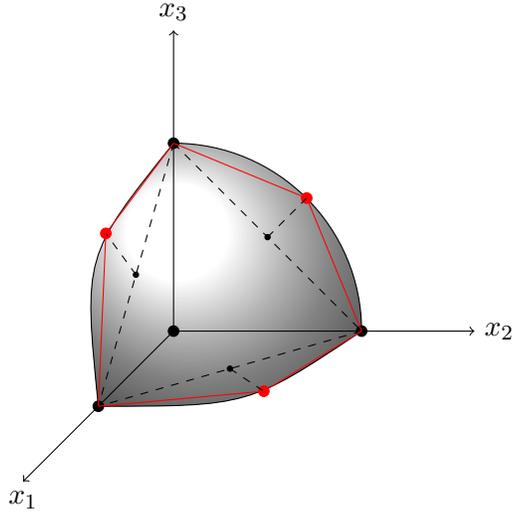

\subsection{The adaptive algorithm}

The numerical simulations are implemented with a C++ library: Adaptive Finite
Element Package (AFEPack). The initial mesh is generated by GMSH\cite{GR09}. The
resulting sparse linear systems are solving by the solver called Eigen. The
simulations are implemented on a HP workstation and are accelerated by using
OpenMP. The a posteriori error estimate from Theorem \ref{thm} is adopted to
generate the indicators in our algorithm. 

The error consists of two parts: the finite
element discretization error $\varepsilon_h$ and the DtN operator truncation
error $\varepsilon_N$ which depends on $N$. Specifically, 
\begin{align}\label{e_N}
\varepsilon_h = \Big( \sum\limits_{K\in\mathcal{M}_h} \eta_K^2 \Big)^{\frac12} =
\eta_{\mathcal{M}_h},\quad \varepsilon_N = \Big(\frac{R'}{R}\Big)^N \|
g\|_{L^2(\partial D)}.
\end{align}
In the implementation, we can choose $R'$, $R$, and $N$ based on
\eqref{e_N} such that finite element discretization error is not contaminated by
the truncation error, i.e., $\varepsilon_N$ is required to be very small
compared with $\varepsilon_h$, for example, $\varepsilon_N \leq 10^{-8}$. For
simplicity, in the following numerical experiments, $R'$ is chosen such
that the scatterer lies exactly in the circle $B_{R'}$ and $N$ is taken to
be the smallest positive integer satisfying $\varepsilon_N \leq 10^{-8}$. 
Table \ref{alg} shows the adaptive finite element algorithm with the DtN
boundary condition for solving the scattering problem.

\begin{table}[htp]
\caption{The adaptive FEM-DtN algorithm.}
\begin{tabular}{ll}\hline
1 & Given a tolerance $\varepsilon > 0$;\\
2 & Choose $R$, $R'$ and $N$  such that $\varepsilon_{N}<10^{-8}$;\\
3 & Construct an initial tetrahedral partition $\mathcal{M}_h$ over
$\Omega$ and compute error estimators;\\
4 & While  $\eta_K >\varepsilon$, do  \\
5 & \qquad mark $K$, refine $\mathcal{M}_h$, and obtain a new mesh
$\hat{\mathcal{M}}_h$.\\
6 & \qquad solve the discrete problem on the $\hat{\mathcal{M}}_h$.\\
7 & \qquad compute the corresponding error estimators;\\
8 & End while.\\ \hline
\end{tabular}
\label{alg}
\end{table}

\subsection{Numerical examples}

We present two numerical examples to illustrate the performance of the proposed
method. In the implementation, the wavenumber is $\kappa =\pi $, which accounts
for the wavelength $\lambda = 2\pi/\kappa=2 $. 

{\em Example 1}. Let the obstacle $D = B_{0.5}$ be the ball with a radius
of 0.5 and $\Omega = B_1\setminus\overline B_{0.5}$ be the computational domain.
The boundary condition $g$ is chosen such that the exact solution is 
\[
u(x) = \frac{e^{{\rm i}\kappa r}}{r},\quad r=|x|.
\]

\begin{figure}[htbp]
 \centering
\includegraphics[width=0.45\textwidth]{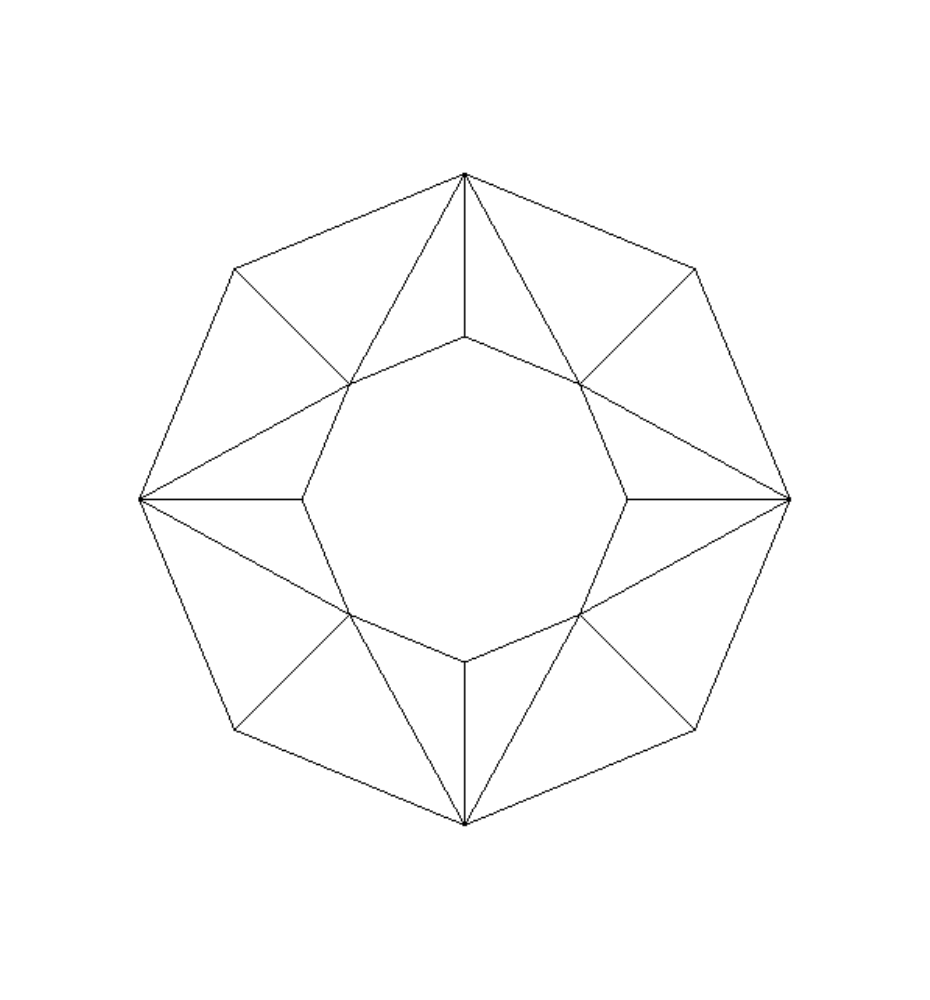}
\includegraphics[width=0.45\textwidth]{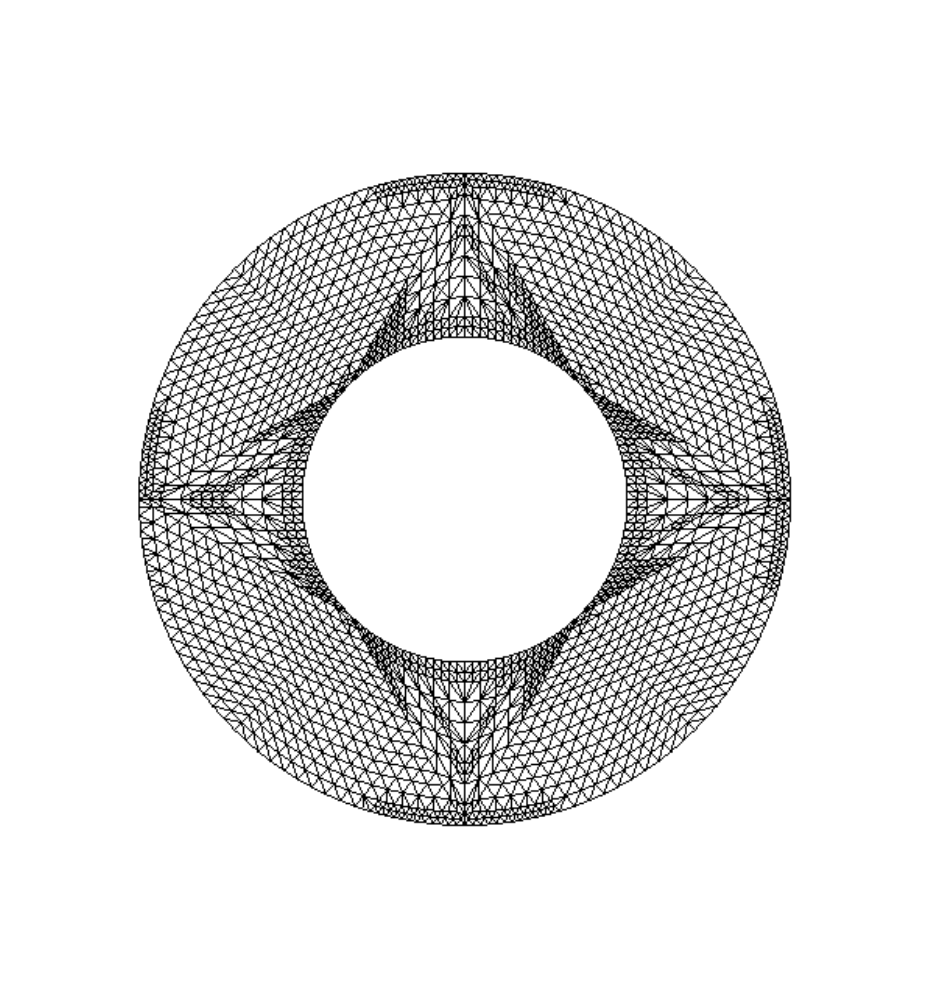}
\caption{Example 1: (left)  initial mesh on the $x_1 x_2$-plane. (right)
adaptive mesh on the $x_1 x_2$-plane.}
\label{figex1_1}
\end{figure}

\begin{figure}[htbp]
 \centering
\includegraphics[width=0.45\textwidth]{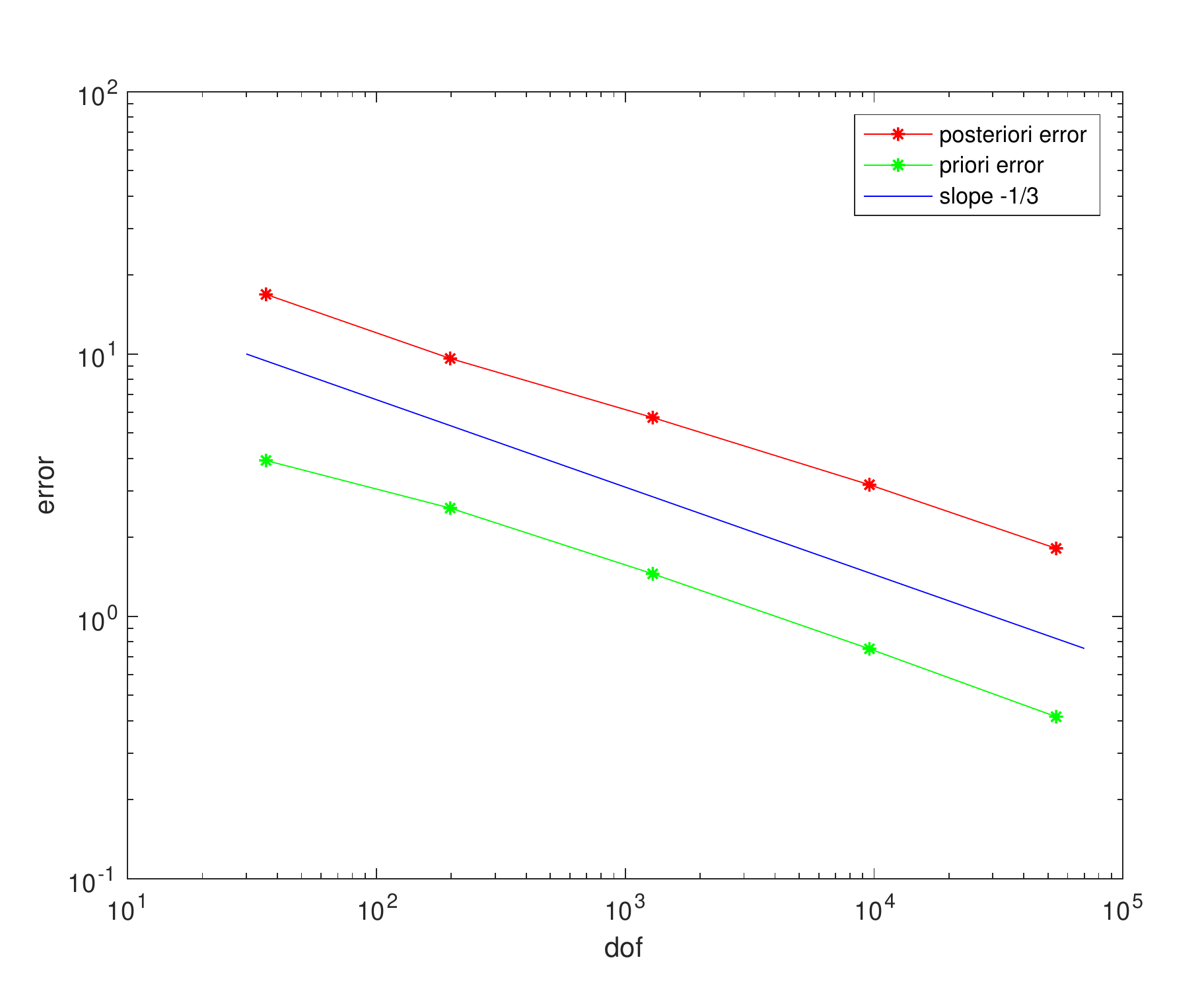}
\caption{Example 1: quasi-optimality of the a priori and a posteriori error
estimates.}
\label{figex1_2}
\end{figure}

The initial mesh and an adaptive mesh is shown in Figure \ref{figex1_1}. Figure
\ref{figex1_2} displays the curves of $\log e_h$ and
$\log \varepsilon_h$ versus $\log {\rm DoF}_h$ for our adaptive DtN method,
where $e_h=\|\nabla(u-u_h^N)\|_{L^2(\Omega)}$ is the a priori error,
$\varepsilon_h$
is the a posteriori error given in \eqref{e_N}, and ${\rm Dof}_h$ denotes the
degree of freedom or the number of nodal points of the
mesh $\mathcal{M}_h$ in the domain $\Omega$. It indicates
that the meshes and associated numerical complexity are quasi-optimal, i.e.,
$\|\nabla(u-u_h^N)\|_{L^2(\Omega)}=\mathcal{O}({\rm
DoF}_h^{-\frac13})$ holds asymptotically.

{\em Example 2}. This example concerns the scattering of the plane wave
$u^{\rm inc} = e^{{\rm i}\kappa x_3}$ by a U-shaped obstacle $D$ which is
contained in the box $\{x\in\mathbb R^3: -0.25 \leq x_1 ,x_2,x_3\leqslant
0.25 \}$. There is no analytical solution for this example and the solution
contains singularity around the corners of the obstacle. The Neumann boundary
condition is set by $g = \partial_\nu u^{\rm inc}$ on $\partial D$. We take $R =
1$, $R' = \frac{\sqrt{3}}{4}$ for the adaptive DtN method. Figure \ref{figex2_1}
shows the cross section of the obstacle and the adaptive mesh of 63898
elements, and the curve of $\log\varepsilon_h$ versus $\log {\rm DoF}_h$. It
implies that the decay of the a posteriori error estimate is $\mathcal{O}({\rm
DoF}_h^{-1/3})$, which is optimal.

\begin{figure}[htbp]
\centering
\includegraphics[width=0.4\textwidth]{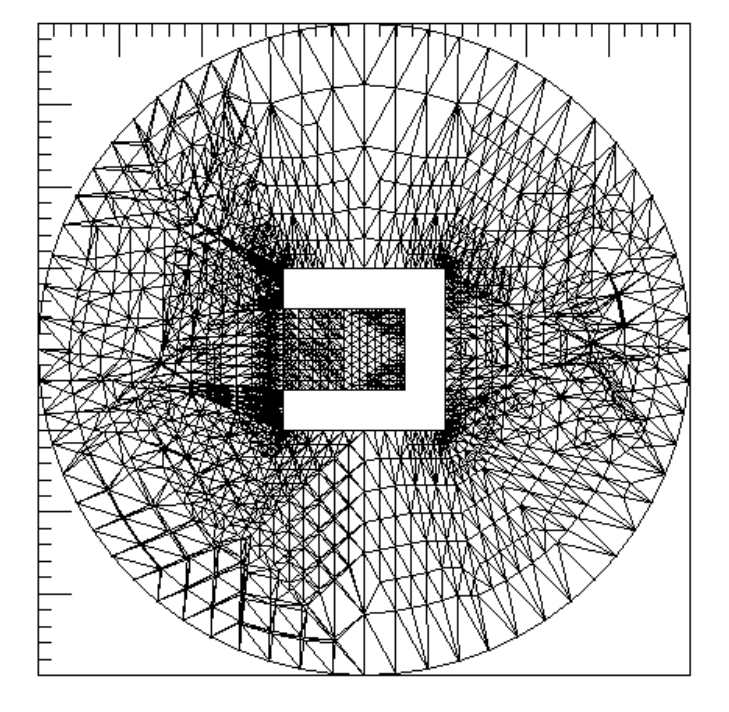}
\includegraphics[width=0.45\textwidth]{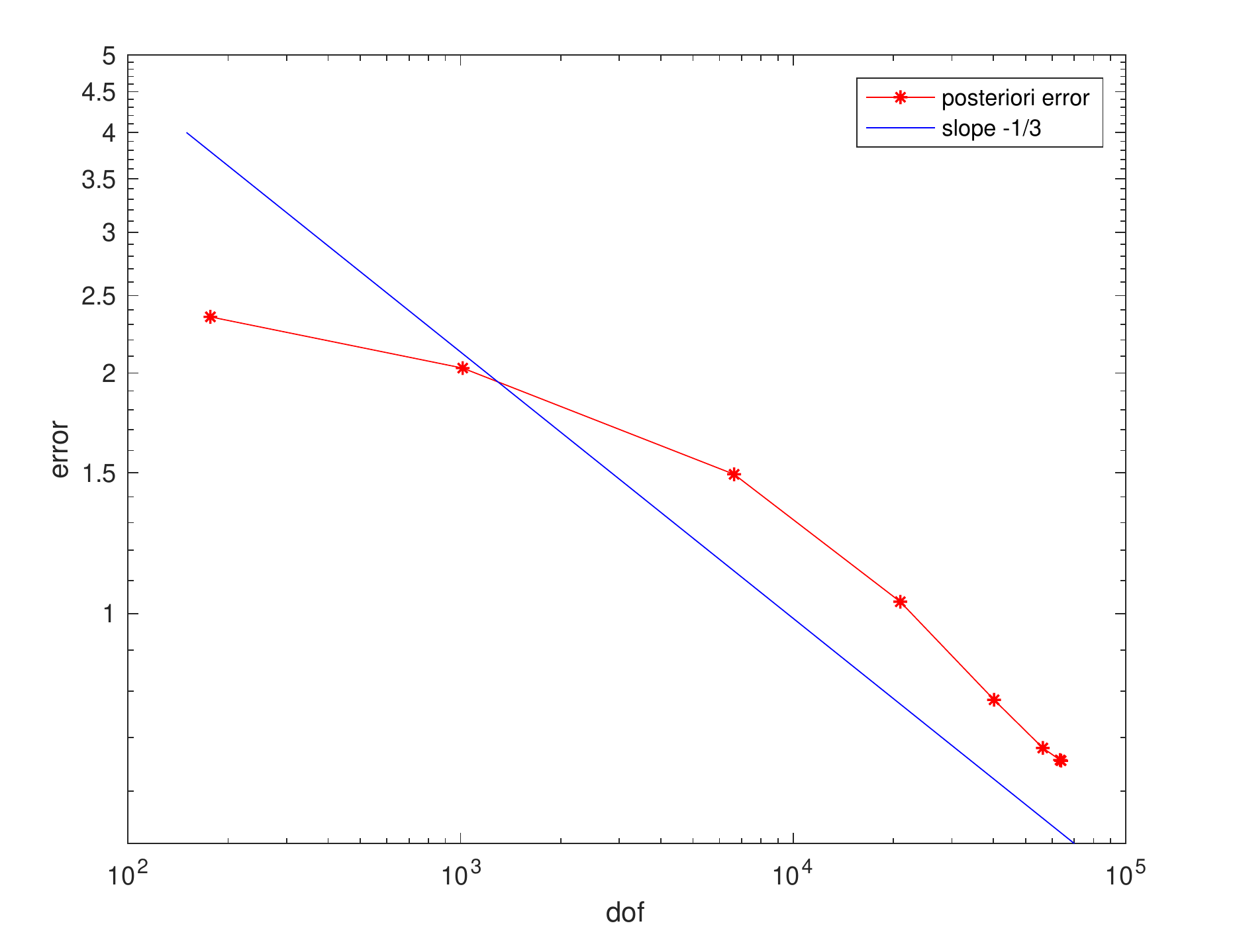}
\caption{Example 2: (left) an adaptively refined mesh with 63898
elements. (right) quasi-optimality of the a posteriori error estimate.}
\label{figex2_1}
\end{figure}

\section{Conclusion}

In this paper, we have presented an adaptive finite element method with the
transparent boundary condition for the three-dimensional acoustic obstacle
scattering problem. The truncated DtN operator was considered for the discrete
problem. A dual argument was developed in order to derive the a posteriori error
estimate. The error consists of the finite element approximation error and the
DtN operator truncation error which was shown to exponentially decay with
respect to the truncation parameter $N$. Numerical results show that the method
is effective to solve the three-dimensional acoustic obstacle scattering
problem. Possible future work is to extend the adaptive FEM-DtN
method for solving the three-dimensional electromagnetic and elastic obstacle
scattering problems, where the wave propagation is governed by the Maxwell
equations and the Navier equation, respectively. We hope to report the progress
on solving these problems elsewhere in the future.

\end{document}